\documentclass [11pt,oneside,a4paper,mathscr]{amsart}
\usepackage{geometry}
\newgeometry{margin=1.1in}

\usepackage{amsfonts, amsmath, amssymb, amsthm,mathtools, stmaryrd}
\usepackage{verbatim}
\usepackage[normalem]{ulem}
\usepackage{tikz-cd}
\usepackage{enumitem}

\usepackage{longtable} 
\usepackage[center]{caption}
\usepackage{diagbox}

\usepackage{array}

\usepackage{hyperref}
\usepackage{cleveref}

\theoremstyle{plain}
\newtheorem{thm}{Theorem}[section]
\newtheorem{cor}[thm]{Corollary}

\newtheorem{prop}[thm]{Proposition}

\numberwithin{equation}{section}

\newtheorem{conjecture}{Conjecture}  
\newtheorem{conj}[conjecture]{Conjecture}  

\theoremstyle{definition}

\newtheorem{example}[thm]{Example}
\newtheorem{lemma}[thm]{Lemma}
\newtheorem{rmk}[thm]{Remark}

\theoremstyle{remark}

\newcommand{\BC}{{\mathbb{C}}}

\newcommand{\BE}{{\mathbb{E}}}

\newcommand{\BG}{{\mathbb{G}}}

\newcommand{\BL}{{\mathbb{L}}}

\newcommand{\BQ}{{\mathbb{Q}}}

\newcommand{\BZ}{{\mathbb{Z}}}

\newcommand{\CC}{{\mathcal C}}

\newcommand{\CL}{{\mathcal L}}

\newcommand{\CO}{{\mathcal O}}

\newcommand{\pt}{{\mathsf{p}}}
\newcommand{\td}{{\mathrm{td}}}

\DeclareFontFamily{OT1}{rsfs}{}
\DeclareFontShape{OT1}{rsfs}{n}{it}{<-> rsfs10}{}
\DeclareMathAlphabet{\curly}{OT1}{rsfs}{n}{it}

\newcommand\Ext{\operatorname{Ext}}
\newcommand\Hom{\operatorname{Hom}}

\newcommand{\p}{\mathbb{P}}

\newcommand\Spec{\operatorname{Spec}}

\newcommand{\Mbar}{{\overline M}}
\newcommand{\vir}{{\text{vir}}}

\newcommand{\Pic}{\mathop{\rm Pic}\nolimits}
\newcommand{\PT}{\mathsf{PT}}

\newcommand{\Pt}{\mathsf{pt}}

\newcommand{\DT}{\mathsf{DT}}
\newcommand{\dt}{\mathsf{dt}}

\newcommand{\GW}{\mathsf{GW}}

\newcommand{\VW}{\mathsf{VW}}

\newcommand{\vw}{\mathsf{vw}}

\newcommand{\QMod}{\mathsf{QMod}}
\newcommand{\Mod}{\mathsf{Mod}}

\newcommand{\Hilb}{\mathsf{Hilb}}

\newcommand{\ch}{\mathsf{ch}}

\newcommand{\SL}{\mathrm{SL}}

\newcommand{\Obs}{\mathrm{Obs}}

\newcommand{\inv}{{\mathrm{inv}}}

\newcommand{\chivir}{{\widehat{\chi}}}
\newcommand{\Exp}{{\mathrm{Exp}}}
\newcommand{\Log}{{\mathrm{Log}}}

\begin{document}
\baselineskip=14.5pt
\title{Towards refined curve counting on the Enriques surface I: K-theoretic refinements}

\author{Georg Oberdieck}

\address{Universit\"at Heidelberg, Institut f\"ur Mathematik}
\email{georgo@uni-heidelberg.de}
\date{\today}

\begin{abstract}
We conjecture an explicit formula for the $K$-theoretically refined  Vafa-Witten invariants of the Enriques surface.
By a wall-crossing argument the conjecture is equivalent to a new conjectural formula for the K-theoretically refined Pandharipande-Thomas invariants of the local Enriques surface.
Evidence for the conjecture is given in several cases.
We also comment on the case of K3 surfaces previously studied by Thomas.	
\end{abstract}

\maketitle

\setcounter{tocdepth}{1} 
\tableofcontents

\section{Introduction}
\subsection{Overview}
Let $Y$ be an Enriques surface and let
\[ X = K_Y, \quad p : X \to Y \]
be the total space of the canonical bundle over it.
In this note we investigate the 
refined Vafa-Witten (VW) invariants of $Y$, denoted
\[ \VW(v) \in \BQ[ t^{1/2}, t^{-1/2} ]. \]
These invariants were conjecturally defined by Thomas \cite{Thomas}
and proven to exist by Liu \cite{Liu}.
They count $K$-theoretically Gieseker semistable compactly supported sheaves $F$ on $X$ with Chern character $\ch(p_{\ast} F) = v \in H^{\ast}(Y,\BZ)$.
Here $t$ is the $K$-theory class of the standard representation of $\BC^{\ast}$, and we let $\BC^{\ast}$ act on $X$ by scaling the fibers.

If $v$ is primitive and the polarization on $Y$ is chosen generic, so that semistability is equal to stability, then the Vafa-Witten invariant is defined by
\[ \VW(v) = \chi(M(v), \widehat{\CO}^{\text{vir}}) \]
where $M(v)$ is the moduli space of compactly supported Gieseker stable sheaves on $X$ with Chern character $v$ and $\widehat{\CO}^{\mathrm{vir}}$ is the Nekrasov-Okounkov twisted virtual structure sheaf, see Section~\ref{subsec:NO twists}.
In most cases this is simply the $\chi_{-y}$-genus of a smooth moduli space of stable sheaves on $Y$, see Example~\ref{example:VW clasical}. 
For general $v$, $\VW(v)$ is defined through Joyce pairs and a wall-crossing formula. 
Since $h^{1,0}(Y) = h^{2,0}(Y) = 0$, the Vafa-Witten invariants of $Y$ can also be viewed as the twisted $K$-theoretic generalized Donaldson-Thomas invariants of $X$.

Under specializing the equivariant parameter $t \mapsto 1$ we obtain
the usual unrefined Vafa-Witten invariants of Tanaka-Thomas \cite{TT1,TT2} denoted by
\[ \VW^{\mathrm{unref}}(v) := \VW(v)|_{t=1}. \]

The unrefined Vafa-Witten invariants of $Y$ were fully computed in \cite{Enriques}. In this paper we will conjecture an explicit formula for the refined Vafa-Witten invariant of the Enriques surface.
Before doing so, let us recall the unrefined computation:
\begin{thm}[{\cite{Enriques}}]
For any effective $(r,\beta,n) \in H^{\ast}(Y,\BZ)$ we have
\[
\VW^{\mathrm{unref}}(r,\beta,n) =
2 \sum_{\substack{k|(r,\beta,n) \\ k \geq 1 \textup{ odd}}}
\frac{1}{k^2} e\left( \Hilb^{\frac{\beta^2 - 2rn - r^2}{2k^2} + \frac{1}{2}}(Y) \right). \]
\end{thm}

Here the Euler characteristic of the Hilbert scheme of points of $Y$ is given by
\[ \sum_{n} e(\Hilb^n Y) q^n = \prod_{n \geq 1} \frac{1}{(1-q^n)^{12}}, \]
and we set $e(\Hilb^n Y) = 0$ if $n$ is fractional.

We then have the following generalization to the refined invariants:
\begin{conj} \label{conj:VW} For any effective $(r,\beta,n) \in H^{\ast}(Y,\BZ)$ we have
\[ \VW(r,\beta,n) = 2 \sum_{\substack{k|(r,\beta,n) \\ k \geq 1 \textup{ odd}}}
\frac{1}{k [k]_t} \chivir_{-t^k}\left( \Hilb^{\frac{\beta^2 - 2rn - r^2}{2k^2} + \frac{1}{2}}(Y) \right), \]
where
\begin{itemize}
\item 
$\chivir_{-t}(X) := (-1)^{\dim(X)} t^{-\dim(X)/2} \chi_{-t}(X)$
denotes the (signed) normalized $\chi_y$-genus of a smooth projective variety $X$, and 
\item $[n]_t := \frac{ t^{n/2} - t^{-n/2}}{t^{1/2} - t^{-1/2}}$
is the quantum integer.
\end{itemize}
\end{conj}

The virtual $\chi_y$-genera of the Hilbert scheme are given by G\"ottsche's formula:
\[
\sum_{n \geq 0} \chivir_{-t}( \Hilb^n(Y))
=
\prod_{m \geq 1} \frac{1}{(1-t^{-1} q) (1-q^m)^{10} (1-t q^m)}.
\]

The formula in Conjecture~\ref{conj:VW} is remarkable, because closed formulas for Vafa-Witten invariants in arbitrary rank are very rare, see \cite{GK1, GK2} for results in rank 2 and 3 and also \cite{Thomas}.
An exception is the K3 surface, where the following formula was recently proven by Thomas \cite{ThomasK3}:
\[
\VW^{K3}(r,\beta,n) = \sum_{k|(r,\beta,n)} \frac{1}{[k]_t^2} \chi_{-t^k}(\Hilb^{\frac{\beta^2-2rn-2r^2}{2k^2}+1}(K3)).
\]
This formula can be viewed as a "multiple-cover formula".
The denominator $1/[k]_t^2$ can be traced back to the use of the reduced virtual class in Pandharipande-Thomas theory for K3 surfaces.
On the other hand, the denominator $1/k [k]_t$ that appears in the Enriques formula is more natural for the unreduced theory and appears also in the multiple-cover behaviour observed in cohomological DT theory, see \cite{Thomas} for a discussion.

Conjecture~\ref{conj:VW} has the following basic consequence for moduli of sheaves on Enriques surfaces for which the author unfortunately does not know a proof.
\begin{conj} \label{conj:Vanishing chi_t}
Let $v=(r,\beta,n) \in H^{\ast}(Y,\BZ)$ be primitive with $r$ even and $2\nmid \beta$. Let $M^Y(v)$ be the (automatically smooth) moduli space of stable sheaves on $Y$ with Chern character $v$ with respect to a generic polarization. Then the $\chi_t$-genus of $M(v)$ vanishes:
\[ \chi_{-t}(M(v)) = 
\sum_{p,q} (-1)^{p+q} h^{p,q}(M(v)) t^p =0. \]
\end{conj}
By a result of Sacca, the moduli spaces in this conjecture are always odd-dimensional Calabi-Yau manifolds \cite{Sacca, Beckmann}.
If it is $1$-dimensional, then $M(v)$ is isomorphic to an elliptic curve, so the vanishing in Conjecture~\ref{conj:Vanishing chi_t} is trivial.
The threefold case was proven by Sacca in \cite{Sacca} (she computed all Hodge numbers).
We give evidence in the 5-fold and 7-fold case in
Section~\ref{subsec:Computations 3 low degree} based on computations of G\"ottsche-Shende.

\subsection{Pandharipande-Thomas theory}
Our main insight into the Vafa-Witten theory of $Y$ is through the
K-theoretic refined Pandharipande-Thomas invariants of the local Enriques surface. These are defined by
\[
\PT_{n,\beta} = \chi( P_{n,\beta}(X), \widehat{\CO}^{\text{vir}}) \in \BQ[ t^{1/2}, t^{-1/2} ]
\]
where $P_{n,\beta}(X)$ is the moduli space of stable pairs $(F,s)$
satisfying $\mathrm{ch}_2(F) = \beta \in H_2(X,\mathbb{Z})$ and $\chi(F) = n$,
and a stable pair $(F,s)$ on $X$ consists  by definition of 
a pure $1$-dimensional sheaf $F$
and a section $s \in H^0(X,F)$ with zero-dimensional cokernel.

Specializing the equivariant parameter we obtain the usual PT invariants
obtained by integrating the virtual class
\[ 
\PT^{\mathrm{unref}}_{n,\beta} := \int_{[ P_{n,\beta}(X) ]^{\vir}} 1 =
\PT_{n,\beta}|_{t=1}. \]
These unrefined invariants are equivalent to the Gromov-Witten invariants of the local Enriques surfaces and were determined in \cite{Enriques} as follows:

Given a power series $f(x_1,\ldots,x_n)$ with vanishing constant term we define a modified plethystic exponential by
\[
\Exp^{(2)}( f(x_1,\ldots, x_n) ) = \exp\left( \sum_{\substack{k \geq 1 \\ k \text{ odd}}} \frac{f(x_1^k, \ldots, x_n^k)}{k} \right).
\]

\begin{thm}[Klemm-Mari\~{n}o formula \cite{KM1}, proven in {\cite{Enriques}}] We have
\begin{equation} \label{unrefined PT}
\sum_{\beta \geq 0} \sum_{n \in \BZ} \PT^{\mathrm{unref}}_{n,\beta} (-p)^n Q^{\beta}
=
\Exp^{(2)}\left( \sum_{\beta > 0} \sum_{r \in \BZ} 2 \omega(r,\beta^2/2) p^r Q^{\beta} \right)
\end{equation}
where the $\omega(r,n) \in \BQ$ are defined by the expansion
\[
\sum_{n,r} \omega(r,n) p^r q^n = \prod_{\substack{m \geq 1 \\ m \text{ odd}}} \frac{1}{(1-p^{-1} q^m)^2 (1-q^m)^{4} (1-p q^m)^2}
			\prod_{m \geq 1} \frac{1}{(1-q^m)^{8}}.
\]
\end{thm}

The conjecture for the refined PT invariants is as follows:


\begin{thm} 
\label{thm:PT} Conjecture~\ref{conj:VW} holds if and only if we have
\[
\sum_{\beta \geq 0} \sum_{n \in \BZ} \PT_{n,\beta} (-p)^n Q^{\beta}
=
\Exp^{(2)}\left( \sum_{\beta > 0} \sum_{r \in \BZ} 2 \Omega(r,\beta^2/2) p^r Q^{\beta} \right)
\]
where  the $\Omega(r,n) \in \BQ[t^{1/2}, t^{-1/2}]$ are defined by the expansion
\begin{small}
\begin{multline*}
\sum_{n,r} \Omega(r,n) p^r q^n = 
\\
\prod_{\substack{m \geq 1 \\ m \text{ odd}}} \frac{1}{
(1-t^{-\frac{1}{2}} p^{-1} q^m) (1-t^{\frac{1}{2}} p^{-1} q^m) (1-t^{-\frac{1}{2}} p q^m) (1-t^{\frac{1}{2}} p q^m)
(1-q^m)^{2} (1-t^{-1} q^m) (1-t q^m)} 
\prod_{m \geq 1} \frac{1}{(1-q^m)^8}
\end{multline*}
\end{small}
\end{thm}

A remarkable aspect of the formula in Theorem~\ref{thm:PT} is that up to taking the modified plethystic exponential $\Exp^{(2)}$, the PT invariants in class $\beta$ only depend on its square $\beta^2$.

The idea behind Theorem~\ref{thm:PT} is to use a wall-crossing formula of Toda that expresses the PT invariants of a K3 fibration in terms of the generalized DT invariants of sheaves supported on fibers of the fibration \cite{Toda}.
Here we may view the local Enriques surface as a K3 fibration over the orbifold $[\BC/\BZ_2]$.
To make this work also for K-theoretic refinements, requires a K-theoretic lift of the wall-crossing formula. This has been recently achieved by Kuhn-Liu-Thimm and represents the main new geometric input in the paper \cite{KLT}.

The same strategy works also for K3 surfaces and allows one to express the reduced PT invariants of the K3 surface in terms of the Vafa-Witten invariants of the K3 surface. This is sketched in Section~\ref{section:K3}. This gives an alternative (but conceptually equivalent) viewpoint on parts of the arguments of Thomas \cite{ThomasK3}.

\subsection{Plan of the paper}
After introducing some background on power series and Jacobi forms in Section~\ref{sec:background}, we review known results on the geometry of moduli spaces of stable sheaves on the Enriques surface.
In Section~\ref{sec:invariants of Enriques} we discuss Toda's wall-crossing formula and how it leads to the proof of Theorem~\ref{thm:PT}, then we do several basic computations to check the conjecture. We consider the fiber classes in Section~\ref{subsec:fiber classes} and low degree curve classes in Section~\ref{subsec:Computations 3 low degree} using results of G\"ottsche-Shende.
In Section~\ref{subsec:HAE} we derive holomorphic anomaly equations for the refined PT series and show they match holomorphic anomaly equations for the refined GW theory of Brini-Sch\"uler \cite{BS}.
In Section~\ref{section:K3} we discuss the K3 case.

\subsection{Future work}
This is the first paper in a series of two papers on the refined curve counting on Enriques surface.
In the sequel \cite{RefinedEnriques2} we will consider properties of the motivic refinement.

\subsection{Acknowledgements}
The main work on this paper was done during a visit of the author to ICTP
in Trieste during February 2024. I thank Lothar G\"ottsche and Alina Marian for hospitality and productive conversations. I also thank Nikolas Kuhn, Oliver Leigh,  Yannik Sch\"uler, Junliang Shen and Richard Thomas for helpful discussions.
The author was supported by the starting grant 'Correspondences in enumerative geometry: Hilbert schemes, K3 surfaces and modular forms', No 101041491 of the European Research Council.

\section{Background}
\label{sec:background}
\subsection{Conventions} \label{subsec:conventions}
Given a power series $f(x_1,\ldots, x_n)$ with zero constant term the plethystic exponential is defined by
\[ \Exp( f(x_1,\ldots, x_n)) = \exp\left( \sum_{\substack{k \geq 1}} \frac{f(x_1^k, \ldots, x_n^k)}{k} \right). \]
It is characterized by $\Exp(x) = 1/(1-x)$ and $\Exp(f+g) = \Exp(f)\Exp(g)$.
The inverse of $\Exp$ is the plethystic logarithm denoted by $\Log$.

The modified plethystic exponential $\Exp^{(2)}$ is given by
\begin{equation} \Exp^{(2)}(f(x)) = \Exp( f(x) - f(x^2)/2). 
	\label{Exp2} \end{equation}
In particular, 
\[ \Exp^{(2)}(x) = \left( \frac{1+x}{1-x} \right)^{1/2}. \]
We let $\Log^{(2)}$ be the inverse to $\Exp^{(2)}$.
(There does not seem to be a reasonable formula expressing $\Log^{(2)}$ in terms of $\Log$.)


We define the quantum integer $[n]_t$ for $n \geq 0$ by
\[ [n]_t := \frac{ t^{n/2} - t^{-n/2}}{t^{1/2} - t^{-1/2}}. \]
In particular, $[0]_t = 0$, $[1]_t = 1$ and  for $n \geq 2$ we have
\[ [n]_t = t^{-(n-1)/2} \sum_{i=0}^{n-1} t^i = t^{-(n-1)/2} + t^{-(n-1)/2+1} + \ldots + t^{(n-1)/2}. \]
\begin{lemma}
\begin{itemize}
\item[(a)] $[ k \ell ]_{t} =  [ k ]_{t} [ \ell ]_{t^k}$
\item[(b)] $\sum_{\ell \geq 1} [ \ell ]_t p^{\ell} = \sum_{\ell \geq 1} \sum_{j=0}^{\ell-1} t^{\frac{-\ell + 1 + 2j}{2}} p^{\ell} = \frac{p}{(1-t^{1/2} p)(1-t^{-1/2}p)}$
\end{itemize}
\end{lemma}

\subsection{Jacobi forms}
Recall the Jacobi theta function
\[ \Theta(p,q) =  (p^{1/2}-p^{-1/2})\prod_{m\geq 1} \frac{(1-pq^m)(1-p^{-1}q^m)}{(1-q^m)^{2}}. \]
Often we drop $q$ from the notation and simply write $\Theta(p)$.
We have the following identities:
\begin{prop} \label{prop:Jac identity K3} We have
\begin{multline*}
\sum_{n \geq 1} [n]_t p^n + \sum_{r \geq 1} \left( [2r]_t q^{r^2} + \sum_{n \geq 1} [n + 2r]_t (p^n + p^{-n}) \right)
\\ =
\frac{p}{(1-t^{1/2} p)(1-t^{-1/2}p)} \prod_{ m \geq 1} \frac{ (1- t q^m)(1-q^m)^2 (1-t^{-1}q^m) }{(1-t^{-1/2} p^{-1} q^m) (1-t^{1/2} p^{-1} q^m) (1-t^{-1/2} p q^m) (1-t^{1/2} p q^m)} \\
= 
\frac{1}{t^{1/2}-t^{-1/2}}
\frac{\Theta(t,q)}{\Theta(t^{1/2}p,q) \Theta(t^{-1/2} p, q)}
\end{multline*}
\end{prop}
\begin{proof}
This is an identity of Zagier, see \cite[Sec.3]{Zagier}.
\end{proof}

The proof of the following two propositions can be argued similar to \cite[Prop.2.2]{Enriques}.
Proposition~\ref{prop:Jac identity Enriques2} will be used only in the sequel \cite{RefinedEnriques2}.

\begin{prop} \label{prop:Jac identity Enriques} We have
\begin{align*}
& \sum_{\substack{r \geq 1 \\ r \text{ odd}}} \left( [r]_t q^{r^2/2} + \sum_{n \geq 1} [n+r]_{t} (p^n + p^{-n}) q^{rn+r^2/2} \right) \\
= & \,
q^{1/2} \prod_{\substack{m \geq 1 \\ m \text{ odd}}} \frac{1}{ (1-t^{-1/2} p^{-1} q^m) (1-t^{1/2} p^{-1} q^m) (1-t^{-1/2} p q^m) (1-t^{1/2} p q^m) } \\
& \quad \quad \times 
\prod_{n \geq 1} (1-t^{-1} q^{2n})(1-q^{2n})^{2} (1-t q^{2n}) \\
= & \, \frac{\Theta(t, q^2)}{(t^{1/2} - t^{-1/2})}
\frac{\Theta(p t^{1/2}, q^2) \Theta(p t^{-1/2}, q^2)}{\Theta(p t^{1/2},q) \Theta(p t^{-1/2},q)}
\cdot \frac{\eta(q^2)^8}{\eta(q)^4}
\end{align*}
\end{prop}

\begin{prop} \label{prop:Jac identity Enriques2} 
\begin{align*}
& \sum_{n \geq 1} [n]_t p^n 
+ \sum_{\substack{r \geq 1 \\ r \textup{ even}}} 
\left( [r]_t q^{r^2/2} + \sum_{n \geq 1} [n+r]_{t} (p^n + p^{-n}) q^{rn+r^2/2} \right) \\
= & \,
\frac{p}{(1-t^{1/2}p) (1-t^{-1/2}p)}
\prod_{m \geq 1} \frac{ (1-t^{-1} q^{2m}) (1-q^{2m})^2 (1-t q^{2m})}{
(1- t^{-\frac{1}{2}} p^{-1} q^{2m}) (1- t^{\frac{1}{2}} p^{-1} q^{2m}) (1-t^{-\frac{1}{2}} p q^{2m}) (1-t^{\frac{1}{2}} p q^{2m})} \\
= & \,
\frac{\Theta(t,q^2)}{(t^{\frac{1}{2}}-t^{-\frac{1}{2}})} \frac{1}{\Theta(t^{\frac{1}{2}} p , q^2) \Theta( t^{-\frac{1}{2}} p, q^2)}
\end{align*}
\end{prop}

\section{Moduli of stable sheaves on (local) Enriques surfaces}
\label{sec:moduli of sheaves on Enriques}
Let $(Y,H)$ be a polarized Enriques surface
and consider a Chern character
\[ v=(r,\beta,n) \in H^{\ast}(Y,\BZ) 
\]
decomposed according to degree.
Define the Mukai square of $v$ by
\[ v^2 := -\chi(v,v) := -\int_{Y} v^{\vee} \cdot v \cdot \td_Y = \beta^2 - r^2 - 2rn. \]

We say that $v=(r,\beta,n)$ is positive if $r>0$, or $r=0$ and $\beta$ is effective, or $r=\beta=0$ and $n>0$.
Let $M^Y_H(v)$ be the moduli space of $H$-Gieseker semistable sheaves $F$ on $Y$ with {\em Chern character} $\ch(F)=v$.\footnote{We work here with the Chern character, since it is always integral. The Mukai vector $v(F) = \ch(F) \sqrt{\td_Y}$ can be half-integral (e.g. $v(\CO_Y)=(1,0,1/2)$).}
Since there is a unique $2$-torsion line bundle on $Y$ (the canonical bundle),
the moduli space decomposes as 
\[ M_H^Y(v) = M_H^Y(v,L) \sqcup M_H^Y(v, L+K_Y), \]
whre $M_H^Y(v,L)$ parametrizes stable sheaves with determinant $L$.
If the rank is odd, the two components are isomorphic,
and interchanged by tensoring with $\omega_Y$.

For primitive vectors $v$ the moduli spaces $M_H^Y(v)$ are very well-behaved:

\begin{thm}[Nuer, Yoshioka {\cite{N1,N2,Yoshioka}}] \label{thm:Nuer Yoshioka}
Let $Y$ be an unnodal\footnote{An Enriques surface is unnodal if it does not contain smooth rational curves (that is $(-2)$ curves). A generic Enriques surface is unnodal. The moduli space of Enriques surfaces is irreducible.} Enriques surface, let $v=(r,\beta,n) \in H^{\ast}(Y,\BZ)$ be positive and primitive, and let $H$ be a generic polarization.
\begin{enumerate}
\item[(i)] The moduli space $M_H^Y(v,L)$ is non-empty if and only if 
\begin{itemize}
\item[(a)] $2 \nmid \gcd(r,\beta)$ and $v^2 \geq -1$,
\item[(b)] $2 \mid \gcd(r,\beta)$ and $v^2 > 0$,
\item[(c)] $2 \mid \gcd(r,\beta)$ and $v^2=0$ (or $(r,\beta)=(0,0)$), and $2|L+\frac{r}{2} K_Y$.
\end{itemize}
\item[(ii)] If $M_H^Y(v,L)$ is non-empty, than it is irreducible. 
\item[(iii)] $M_H^Y(v,L)$ is of dimension $v^2+1$, unless in case (c) where it is of dimension $2$.
\item[(iv)] $M_H^Y(v,L)$ is smooth with torsion canonical bundle in case (a).
\end{enumerate}
\end{thm}

\begin{rmk}
Let $\pi: S \to Y$ be the covering K3 surface.
For $v=(r,\beta,n)$ primitive, the condition $2 \nmid \gcd(r,\beta)$ is equivalent to $\pi^{\ast}(v)$ primitive.
\end{rmk}

We also state a result on the birationality type
for the moduli spaces in case (a):

\begin{thm}[Beckmann, Nuer, Yoshioka, Sacca] \label{thm:BNYS}
Let $v=(r,\beta,n)$ be a primitive Mukai vector on an Enriques surface $Y$ satisfying $2 \nmid \gcd(r,\beta)$. Let $H$ be a generic polarization.
\begin{enumerate}
\item[(i)] If $r$ is odd and $v^2 > 0$, then $M_H^Y(v)$ is birationally equivalent to the Hilbert scheme of $k$ points on $Y$, where $k=(v^2+1)/2$.
\item[(ii)] If $r$ is even and $v^2 > 0$, then $M_H^Y(v)$ is birationally equivalent to $M_H^Y(0,\beta,1)$ for any primitive and effective $\beta$ satisfying $v^2=\beta^2$.
Moreover, $M=M_H^Y(v)$ is a projective Calabi-Yau manifold, i.e. $\omega_M = \CO_M$ and $h^{p,0}(M) = 0$ for $p \neq 0, \dim(M)$.
\end{enumerate}
\end{thm}
\begin{proof}
See \cite[Thm. 4.7 and Prop.4.8]{Beckmann} or \cite{NY}.
The Calabi-Yau part in (ii) follows from Sacca \cite{Sacca}.
\end{proof}

\begin{rmk}
By Kontsevich's work on motivic integration, the Hodge polynomial of $M_H^Y(v)$ in case (i) of Theorem~\ref{thm:BNYS} is hence equal to the Hodge polynomial of the Hilbert scheme of points and thus determined by G\"ottsche's formula.
In case (ii) it equals to the Hodge polynomial of $M(0,\beta,1)$, which is unknown so far.
\end{rmk}

We consider now the total space $X = K_Y$.
Let us denote with $M_H(v)$ the moduli space of compactly supported  $H$-Gieseker semistable sheaves $F$ on $X=\mathrm{Tot}(\omega_Y)$ with $\ch(p_{\ast} F) = v$, where $p : X \to Y$ is the projection.
We often drop $H$ from the notation.

\begin{lemma} \label{lemma:sheaves on X and Y comparision}
	Let $Y$ be a generic Enriques surface, $v=(r,\beta,n)$ primitive and $H$ generic. If $2 \nmid \gcd(r,\beta)$,
	then the natural inclusion $M^Y_H(v) \subset M_H(v)$ given by pushforward along the zero section is an isomorphism.
\end{lemma}
\begin{proof}
	We have $K_Y = \Spec( \oplus_{i \geq 0} \omega_Y^{-i})$. The inclusion
	$\oplus_i \CO_Y \cong \oplus \omega_Y^{-2i} \subset \oplus_{i \geq 0} \omega_Y^{-i}$ gives a map $h : K_Y \to Y \times \BC \to \BC$. We have
	\[ h^{-1}(a) = \begin{cases}
		\tilde{Y} & \text{ if } a \neq 0 \\
		2Y & \text{ if } a=0
	\end{cases}
	\]
	where $\tilde{Y} \to Y$ is the covering K3 surface.
	
	If $F \in M_H(v)$, then since $v$ is primitive and $H$ is generic, $F$ is stable, so must be supported on a fiber $h^{-1}(a)$.
	If $a \neq 0$, then $v = \pi_{\ast} v'$ for $v' \in H^{\ast}(\tilde{Y},\BZ)$. In particular, $r$ must be even and since $Y$ is generic, we have $\Pic(\tilde{Y}) = \pi^{\ast} \Pic(Y)$, which shows that $2|\beta$.
	If we assume $r$ is odd or $\beta$ not $2$-divisible, we hence get that $F$ is set-theoretially supported on $h^{-1}(0)=2Y$.
	There is a canononical $s \in H^0(X, p^{\ast} \omega_Y)$ which vanishes precisely at the zero section $Y \subset X$.
	Multiplying by $s$ gives a morphism $p^{\ast}\omega_Y \otimes F \to F$. Since its square is zero, it can not be an isomorphism. So by stability it must vanish, and $F$ is supported on the zero section.
\end{proof}

We can also analyse the situation infinitesimally.
\begin{lemma}
	If $F$ be a coherent sheaf on $Y$, then we have the following long exact sequence
	of $\BC^{\ast}$-equivariant vector spaces:
	\begin{align*}
		0 \to & \Ext^1_Y(F, F) \to \Ext^1_X( \iota_{\ast} F, \iota_{\ast} F) \to \Hom(F, F \otimes \omega_Y) \otimes t \\
		\to & \Ext^2_Y(F,F) \to \Ext^2_X(\iota_{\ast} F, \iota_{\ast} F) \to \Ext^1(F, F \otimes \omega_Y) \otimes t \to 0,
	\end{align*}
where $\BC^{\ast}$ acts on $X$ by scaling the fibers.
\end{lemma}
\begin{proof}
	See for example \cite[Sec.4.2]{QuasiK3} for similar results.
\end{proof}

In particular, the inclusion $\Ext^1_Y(F, F) \to \Ext^1_X( \iota_{\ast} F, \iota_{\ast} F)$ can only fail to be an isomorphism if there is a non-zero map $F \to F \otimes \omega$. If $F$ is also stable, this map has to be an isomorphism, in which case $F$ is a pushforward of a sheaf on $\tilde{Y}$.

\section{K-theoretic invariants of the Enriques surface}
\label{sec:invariants of Enriques}
\subsection{Nekrasov-Okounkov twists}
\label{subsec:NO twists}
Following \cite[Sec.2]{Thomas}, let $M$ be a quasi-projective scheme with a $\BC^{\ast}$-action with projective fixed locus $M^T$, and a $T$-equivariant symmetric perfect obstruction theory $E^{\bullet} = [E^{-1} \to E^0] \to \BL_M$.
Let $K^{\mathrm{vir}} = \det(E^{\bullet})$ be the virtual canonical bundle
and let $\CO_M^{\mathrm{vir}}$ be the virtual structure sheaf.

As explained by Nekrasov and Okounkov \cite{NO} the natural $K$-theoretic invariants to consider in this geometry are the equivariant Euler-characteristics $\chi(M, \widehat{\CO_M^{\vir}})$,
where 
\[ \widehat{\CO_M^{\vir}} =  \CO_M^{\vir} \otimes K_{\vir}^{1/2} \]
is a twisted structure sheaf and $K_{\vir}^{1/2}$ is a square root of $K_{\vir}$.
However, the square root does not have to exist in general and may not be unique.
Instead to obtain a unambiguous definition,
one defines the left hand side by localization.
Namely after restriction to the fixed locus,
the virtual canonical bundle 
$K_{\vir}|_{M^T}$ admits a {\em canonical} square root $K_{\vir}^{1/2}|_{M^T}$, see \cite[Prop.2.6]{Thomas}. So one defines $\chi(M, \widehat{\CO^{\vir}})$ by the (virtual) $K$-theoretic localization formula \cite{Thomas}:\footnote{The localization formula in $K$-theory is the following:
	If a torus $T$ acts on a smooth variety $M$ with proper fixed locus $M^T$ with normal bundle $N$, and $E$ is an $T$-equivariant sheaf, then
	\[ \chi(M,E) = \chi\left( M^{T}, \frac{E|_{M^T} }{\bigwedge^{\bullet} N^{\vee}} \right). \]}
\[
 \widehat{\CO_M^{\vir}} := 
  \chi_{t}\left( M^T, \frac{\CO_{M^T}^{\vir}}{\wedge^{\bullet}( N^{\vir} )^{\vee} } \otimes K_{\mathrm{vir}}^{1/2}|_{M^T} \right),
\]
where $N^{\vir}$ is the virtual normal bundle.
By \cite[Prop.2.22]{Thomas},
$\widehat{\CO_M^{\vir}} \in \BQ(t^{1/2})$ is a rational function in $t^{1/2}$ with poles at roots of unity and at zero, but no pole at $t=1$,
invariant under $t \mapsto 1/t$.

\begin{example}
Assume $M$ is a smooth connected projective variety with trivial torus action and an equivariant symmetric perfect obstruction theory.
The obstruction theory can then be written as $E^{\bullet} = [ T_M \otimes t^{-1} \xrightarrow{0} \Omega_M ]$. 
The canonical square root of the virtual canonical bundle is
\[ (K_M^{\vir})^{1/2} = \omega_M \otimes t^{\frac{\dim(M)}{2}}. \]
The virtual structure sheaf becomes
\[ \CO_M^{\vir} = \bigwedge^{\bullet} \Obs^{\vee} = \bigwedge^{\bullet} T_M \otimes t^{-1}, \]
where 
$\Obs = T_M^{\ast} \otimes t$ is the obstruction bundle and
$\wedge E = \sum_{i \geq 0} (-1)^i \wedge^i E$, see \cite[Sec.3.2]{FG}.

The shifted virtual structure sheaf is hence
\[ \widehat{\CO}_M^{\vir} = \CO_M^{\vir} \otimes K_{\vir}^{1/2} = \sum_{i=0}^{\dim(M)} (-1)^i \left( \bigwedge^{\dim(M)-i} \Omega_M \right) \otimes t^{\frac{\dim(M)}{2}-i}. \]
Hence we obtain
\[
\chi(M, \widehat{\CO}_M^{\vir}) = (-1)^{\dim(M)} t^{-\frac{\dim M}{2}} \sum_{j \geq 0} \chi(M, \wedge^j \Omega_M) (-t)^j
= \chivir_{-t}( M ).
\]
\end{example}

\begin{example}
\label{example:VW clasical}
In particular, if $Y$ is an Enriques surface, $v=(r,\beta,n)$ primitive, and $2 \nmid \gcd(r,\beta)$, then
$M^Y(v) \subset M(v)$ is an isomorphism (Lemma~\ref{lemma:sheaves on X and Y comparision}), $M^Y(v)$ is smooth (Theorem~\ref{thm:Nuer Yoshioka}) and hence the Vafa-Witten invariant satisfies:
\[ \VW(v) = \chivir_{-t}(M(v)). \]
\end{example}

\subsection{Two structure theorems} \label{subsec:basic results}
We consider the Nekrasov-Okounkov twisted Vafa-Witten and Pandharipande-Thomas invariants of the local Enriques surface $X=K_Y$,
denoted
$\VW(v)$ for $v \in H^{\ast}(Y,\BZ)$ and
$\PT_{n,\beta}$ for $\beta \in H_2(Y,\BZ)$ respectively.

The unrefined version of these invariants were studied and computed in \cite{Enriques}. A key role was played there by two structure results.
The first is a relationship between PT and VW invariants, the second an independence statement for VW invariants.
The key insight of this section is that both statements lift also to NO-refined invariants. 

We begin with the 
refined version of \cite[Thm 5.14]{Enriques}.

\begin{thm}(Toda's equation) \label{thm:Todas formula}
\begin{align*}
\sum_{\beta \geq 0} \sum_{n \in \BZ} \PT_{n,\beta} (-p)^n Q^{\beta}
=& \prod_{\substack{r \geq 0 \\ \beta > 0 \\ n \geq 0 }} \exp\left( (-1)^{r-1} [n+r]_t \VW(r,\beta,n) Q^{\beta} p^n \right) \\
 \times&  \prod_{\substack{r>0 \\ \beta>0 \\ n>0}} \exp\left( (-1)^{r-1} [n+r]_t \VW(r,\beta,n) Q^{\beta} p^{-n} \right)
\end{align*}
\end{thm}
\begin{proof}
In \cite{Toda} Toda constructed a path of certain stability conditions, so that the stable pbjects at the start point are the stable pair invariants, and the stable objects at the end are a certain line bundles (which are easy to count). Along the path, the moduli space changes along walls which arise from two-dimensional compactly supported sheaves on the local Enriques surface.
As soon as one has a wall-crossing formula in the style of Joyce theory \cite{Joyce}, this construction gives rise to the product expansion above, the factors precisely corresponding to the wall-crossing terms.

For unrefined invariants the wall-crossing formula was obtained by considering the compact geometry of the Enriques Calabi-Yau threefold
and performing the wall-crossing there, see \cite{Enriques} (using the work of Joyce-Song \cite{JS}). These implied the results for the invariants of the local Enriques surface by a degeneration argument. For K-theoretic invariants this strategy can not succeed.
Instead, we use here the recent work of Kuhn-Liu-Thimm \cite{KLT} which exactly prove such a K-theoretic wall-crossing formalism in instances which incclude the local Enriques. It works parallel to the unrefined case, except that the  wall-crossing factor $\chi(E_1,E_2)$ between two sheaves has to be replaced by the corresponding quantum integer $[ \chi(E_1, E_2) ]_t$. The outcome is the formula above.\footnote{I thank Nikolas Kuhn for discussions on this point. We also refer to Thomas \cite{ThomasK3} for a further discussion of refined wall-crossing formulas.}
\end{proof}

In an identical fashion, we obtain the refined version of \cite[Thm 5.4]{Enriques} (we use a different but equivalentl convention of the notion of divisibility, that is easier to use).

\begin{thm} \label{prop:DT dependence}
Let $v=(r,\beta,n) \in H^{\ast}(Y,\BZ)$.
\begin{enumerate}
\item[(i)] The invariant $\VW(v)$ does not depend on the choice of polarization used to define it.
\item[(ii)] The invariant $\VW(v)$ depends upon $v$ only through
\begin{itemize}
\item the Mukai square $v^2 = \beta^2 - 2rn - r^2$
\item the divisibility $\gcd(r,\beta,n)$
\item the type, defined as
\[
\begin{cases}
\text{ even } & \text{ if } \frac{r}{\gcd(r,\beta,2n)},\frac{2n}{\gcd(r,\beta,2n)} \text{ are both even } \\
\text{ odd } & \text{ otherwise }.
\end{cases}
\]
\end{itemize}
In other words, $\VW(v) = \VW(v')$ if $v,v'$ have the same square, divisibility and type.
\end{enumerate}
\end{thm}
\begin{proof}
If two vectors $v,v'$ have the same square, type and divisiblity,
then they are in the same orbit under the action of the derived monodromy group on $H^{\ast}(Y,\BZ)$. Hence by wall-crossing, $\VW(v) = \VW(v')$, see \cite[Proof of Thm 5.4]{Enriques}.
\end{proof}

To get an idea how to think of the different type of vectors that can appear here, we can look at primitive vectors $v=(r,\beta,n)$. There are three different types of them
(below we let $\pi : S \to Y$ be the covering K3 and let $s,f \in H^2(Y,\BZ)$ be effective classes with $s \cdot f = 1$)
\begin{enumerate}
\item[(i)] \underline{$v$ has odd type and $v^2$ is odd.} \\
Since $v^2$ is odd, one has rank $r$ odd.
This implies that $\pi^{\ast}(v)$ is primitive (since if $\pi^{\ast}(v)$ is imprimitive, then it has divisibility $2$).
A prototypical example in this orbit is
$v=(1,0,-n)$ which is the case of the Hilbert scheme of points. 
\item[(ii)] \underline{$v$ has even type and $v^2$ is even.} \\
One has that $\pi^{\ast}(v)$ is primitive, $r$ is even and $\beta$ is not divisible by $2$, see \cite[Rmk5.8]{Enriques}.
The prototypical example is $(0,\beta,0)$ for $\beta$ a primitive class,
or $(0,\beta,1)$ where $2 \nmid \beta$.
\item[(iii)] \underline{$v$ has odd type and $v^2$ is even.}\\
Here the rank is even, so we must have that $\gcd(r,\beta,2n)$ is divisibile by $2$, hence that $\beta$ is divisible by $2$.
Moreover, $v^2 \equiv 0(8)$ (since if $r=2r'$, then $v^2 \equiv -4 r' (r'+n)$ mod 8, and $r',r'+n$ can not be both odd, since $v$ is primitive).
Moreover, $\pi^{\ast}v$ is divisible by $2$ in $H^{\ast}(X,\BZ)$.
Examples are $v=(0,0,1)$ (of square 0), or $v=(2,0,2 \ell + 1)$ of square $-8-8 \ell$, or $v=(0,2 \beta,1)$ for any effective $\beta$.
\end{enumerate}

Note that the cases (ii) and (iii) can have the same squares, but are of different type, so a priori the DT invariant even for primitive $v$ does not only depend on the square.

\begin{cor} \label{cor:vanishing}
The invariant $\VW(v)$ vanishes whenever $(v/\mathrm{div}(v))^2 < -1$.
\end{cor}
\begin{proof}
Let $w=v/\mathrm{div}(v)$. Then $w$ is primitive. Assume $w^2<-1$.
If it is of type (i) above, then $w$ is equivalent to $(1,0,k)$ for some $k>0$, and so
$\DT(v)=\DT(m(1,0,k)) = 0$, since there are no semi-stable sheaves in that case.\footnote{Why exactly?}
If it is of type (ii) then $\DT(v)=\DT(m(0,s+df,0))$ where $d<0$. Since $s+df$ is not effective in this case, there are no semi-stable sheaves and the DT invariant vanishes.
Finally, for type (iii) $v$ is equivalen to $(0,2(s+df),1)$ for $d<0$, so again there are no semi-stable sheaves and the invariant vanishes.
\end{proof}

\subsection{Numerical consequences}
Our next goal is to deduce Theorem~\ref{thm:PT} from the above.
Define the invariants $\vw(r,\beta,n) \in \BQ[t^{1/2}, t^{-1/2}]$ recursively by
\[ \VW(r,\beta,n) = \sum_{\substack{k | (r,\beta,n) \\ k \text{ odd}}} \frac{1}{k [k]_t} \vw\left( r/k, \beta/k, n/k \right)|_{t \mapsto t^k}. \]
Moreover, define $\Pt_{n,\beta} \in \BQ[t^{1/2}, t^{-1/2}]$ for $\beta > 0$ by the expansion
\[ 
\Log^{(2)}\left( \sum_{n,\beta} \PT_{n,\beta} (-p)^n Q^{\beta} \right)
= \sum_{n,\beta} \Pt_{n,\beta} (-p)^n Q^{\beta}, \]
where the modified plethystic logarithm $\Log^{(2)}$ was defined in Section~\ref{subsec:conventions}.

Applying $\Log^{(2)}$ to Theorem~\ref{thm:Todas formula} yields for all $\beta > 0$:
\begin{equation} \label{toda small}
\begin{aligned}
\sum_{n} (-p)^n \Pt_{n,\beta}
 = & \sum_{r,n \geq 0} (-1)^{r-1} \dt(r,\beta,n) [n+r]_t p^n \\
 +&  \sum_{r,n > 0} (-1)^{r-1} \dt(r,\beta,n) [n+r]_t p^{-n}.
 \end{aligned}
\end{equation}
where
we used the identity $[k n]_t/[k]_t = [n]_{t^k}$.

\begin{proof}[Proof of Theorem~\ref{thm:PT}]
If Conjecture VW holds, then a short computation gives
\[
\vw(r,\beta,n) = \left[ 2 \prod_{m \geq 1} \frac{1}{(1-t^{-1} q^m) (1-q^m)^{10} (1-t q^m)} \right]_{q^{\beta^2/2 - rn - r^2/2+1/2}}. \]
In particular, $\vw(v)$ only depends only on the Mukai square $v^2=\beta^2-2rn -r^2$.
Moreover, $\vw(r,\beta,n) =0$ whenever  $r$ is even. Write
$\vw(v) = a(v^2/2+1/2)$.
Our assumption then gives
\begin{align*} 
\sum_n q^n a(n) 
& = 2 \prod_{m \geq 1} \frac{1}{(1-t^{-1} q^m) (1-q^m)^{10} (1-t q^m)} \\
& = \frac{2 (t^{1/2} - t^{-1/2}) q^{1/2}}{ \Theta(t,q) \eta(q)^{12}}.
\end{align*}

Since $\dt(v)$ depends only on $v^2$, by \eqref{toda small} also $\Pt_{n,\beta}$ depends on $\beta$ only through $\beta^2$. Hence we need to compute $\PT_{n,\beta}$ only in the case $\beta=\beta_d$ for some curve class $\beta_d$ with $\beta_d^2=2d$.
This is easily done:
\begin{align*}
& \sum_{n} (-p)^n q^d \Pt_{n,\beta_d} \\
\overset{\eqref{toda small}}{=} & 
\sum_{r,d,n} (-1)^{r-1} \dt(r,\beta_d,n) q^d p^n + \sum_{r,n>0, d} (-1)^{r-1}  \dt(r,\beta_d,n) q^d p^{-n} \\
= & \left( \sum_{n} a(n) q^n \right)
q^{-1/2}
\sum_{\substack{r \geq 1 \\ r \text{ odd}}} \left( [r]_t q^{r^2/2} + \sum_{n \geq 1} [n+r]_{t} (p^n + p^{-n}) q^{rn+r^2/2} \right) \\
\overset{\text{Prop.\ref{prop:Jac identity Enriques}}}{=} & 
\left( \sum_{n} a(n) q^n \right)
q^{-1/2} \cdot 
\, \frac{\Theta(t, q^2)}{(t^{1/2} - t^{-1/2})}
\frac{\Theta(p t^{1/2}, q^2) \Theta(p t^{-1/2}, q^2) \eta(q^2)^8}{\Theta(p t^{1/2},q) \Theta(p t^{-1/2},q) \eta(q)^4} \\
= & 
\frac{2 \Theta(t, q^2)}{\Theta(t,q) \eta^{12}(q)} \cdot
\frac{\Theta(p t^{1/2}, q^2) \Theta(p t^{-1/2}, q^2) \eta(q^2)^8}{\Theta(p t^{1/2},q) \Theta(p t^{-1/2},q) \eta(q)^4}.
\end{align*}
This gives the formula for the PT invariants in Theorem~\ref{thm:PT}.

Conversely,
the arguments 
of part (ii) of the proof of \cite[Proposition 5.16]{Enriques} shows that
\eqref{toda small} is in fact invertible: knowing $\Pt_{n,\beta}$ determines $\vw(v)$ uniquely, so given $\Pt_{n,\beta}$ there is at most one solution to \eqref{toda small}, which then must be given by Conjecture~\ref{conj:VW}.
\end{proof}

\subsection{Motivation for Conjecture~\ref{conj:VW}}
Let $Y \to \p^1$ be an elliptic fibration with half-fiber $f$ and $2$-section $s$. We have that $s^2=f^2=0$ and $s \cdot f=1$, and since we assume $Y$ generic, we can assume $s,f$ are classes of smooth rigid elliptic curves.

We can characterize Conjecture~\ref{conj:VW} as follows:
\begin{lemma} \label{lemma:equivalent conj}
Conjecture~\ref{conj:VW} is equivalent to the following two statements:
\begin{enumerate}
\item[(i)] $\vw(v)$ for $v=(r,\beta,n)$ only depends on the Mukai square
$v^2 = \beta^2 - 2rn - r^2$,
\item[(ii)] $\chi_{-t}(M^Y(0,s+df,0))=0$ for all $d \geq 0$.
\end{enumerate}
\end{lemma}
\begin{proof}
If (i) holds, then $\vw(v)$ for $v^2$ odd is determined by the case of Hilbert scheme of points where $v=(1,0,-n)$, so is given by G\"ottsche's formula. If $v^2$ is even, then $\vw(w)$ is determined by the invariant of the moduli spaces $M(0,\beta_d,0)$, where $\beta_d=s+df$.
By Lemma~\ref{lemma:sheaves on X and Y comparision} $M(0,\beta_d,0) = M^Y(0,\beta_d,0)$, and by Theorem~\ref{thm:Nuer Yoshioka} these are smooth of dimension $2d+1$, so their Vafa-Witten invariants are non-zero multiples of their $\chi_{-t}$-genus.
\end{proof}

Property (i) in the above lemma is quite drastic since it says that 
there should be no dependence of $\vw(v)$ on the divisibility or the type.
This holds for the unrefined invariants.
For the refined invariants it is motivated by the following evidence:

\begin{prop} \label{prop:VW computations} We have
\begin{itemize}
\item[(a)] $\vw(0,0,n) = 0$ for all $n \geq 1$.
\item[(b)] $\vw(0,f,0) = 0$.
\item[(c)] $\vw(r,df ,0) = \begin{cases} 2 & \text{ if } r=1 \\
0 & \text{ if } r \geq 2. \end{cases}$ for all $d \geq 1$.
\end{itemize}
\end{prop}

In particular, although $(0,0,1)$ and $(0,f,0)$ have different types and the corresponding moduli spaces are very different ($M(0,0,1) \cong X$ while $M(0,f,0)$ is isomorphic to two copies of an elliptic curve),
they have the same invariant: $\vw(0,0,1) = \vw(0,f,0) = 0$.

The case (c) shows that the divisibility should not play a role.
%
%

\begin{rmk}
We could define Nekrasov-Okounkov twisted "BPS classes" $\Omega^{\mathrm{NO}}(v)$ by
\[ \VW(v) = \sum_{\substack{k|v \\ k \geq 1}} \frac{1}{k [k]_{t}} 
\Omega^{\mathrm{NO}}(v/k)|_{t \mapsto t^k}. \]
For compact Calabi-Yau threefolds one would expect these to have integrality propertis \cite{JS}. Here however this fails, e.g.
$\Omega^{\mathrm{NO}}(2,0,0) = -1/[2]_{t}$ with unrefined limit $-1/2$. \qed
\end{rmk}
\subsection{Computation I: Fiber classes} \label{subsec:fiber classes}
\begin{prop} \label{prop:fiber classes NO} We have
\[ \sum_{d \geq 0} \PT_{df,0} q^d = \prod_{m \geq 1} \frac{(1-q^{2m})}{(1-q^m)^{2}}. \]
\end{prop}
\begin{proof}
Let $\BZ_2$ act on $\BC^2$ by $(x,y) \mapsto (-x,-y)$.
Let $T = \BG_m^2$ act on $\BC^2$ by scaling the coordinates.
This $T$ action induces an action on $\BC^2/\BZ_2$,
and hence on its crepant resolution $T^{\ast} \p^1$.
Let $\mathbf{a}$ be a torus weight of $T$. Let $q=q_0q_1$. Let $\Hilb^{m_0,m_1}([\BC^2/\BZ_2])$ be the Hilbert scheme of points of $[\BC^2/\BZ_2]$ parametrizing $\BZ_2$-equivariant zero-dimensional subschemes $z \subset \BC^2$ with $H^0(\CO_z)$
the direct sum of $m_0$ copies of the trivial and $m_1$ copies of the non-trivial irreducible $\BZ_2$ representation \cite{BG}. Let also 
\begin{equation}
 \Hilb^m([\BC^2/\BZ_2]) = \bigsqcup_{m_0+m_1=m} \Hilb^{m_0,m_1}([\BC^2/\BZ_2]). \label{erfer3}
 \end{equation}

Since $\Hilb^{m_0,m_1}([\BC^2/\BZ])$ is equivariantly deformation
equivalent to $\Hilb^{m_0-(m_0-m_1)^2}(T^{\ast} \p^1)$, see \cite{BG},
and by Okounkov's \cite{Okounkov} computation of $\chi( \Hilb^n(T^{\ast} Y), \mathrm{taut})$, we have
\begin{equation} \label{bla} \sum_{m_0,m_1} \chi(\Hilb^{m_0,m_1}( [\BC^2/\BZ_2], \bigwedge^{\bullet} (\Omega \otimes \mathbf{a}) ) q_0^{m_0} q_1^{m_1}
=
\left( \sum_{m \in \BZ} q_1^m q^{m^2} \right)
\Exp\left( \chi( T^{\ast}\p^1, \bigwedge^{\bullet}(\Omega \otimes \mathbf{a}) ) \right). \end{equation}

If $t_1,t_2$ are the tangent weights at the origin of $\BC^2$, then 
at the $2$ fixed points of $T^{\ast} \p^1$ we have tangent weights
$t_1^{-1} t_2, t_1^2$ and $t_1 t_2^{-1}, t_2^2$ (To see this, note that $\BC^2/\BZ_2$ is the spectrum of $\BC[x^2,xy,y^2] = \BC[a,b,c]/(b^2-ac)$, and $T^{\ast} \p^1$ is obtained by blowing up the origin.)
This allows one to compute the right hand side of \eqref{bla} by $K$-theoretic equivariant localization.

In particular, specializing the weight to $\mathbf{a}=t_1$ we obtain
\[
\sum_{m} \chi( \Hilb^m( [\BC^2/\BZ_2]), \bigwedge( \Omega \otimes t_1) t_1^{-\dim/2} ) q^n
= \prod_{m \geq 1} \frac{(1-q^{2m})^2}{1-q^m} \Exp\left( \frac{q^2}{1-q^2} \cdot \frac{ t_1 t_2 + 2 t_1 + 2 t_2 + 1 }{ (t_1 + 1)(t_2+1) } \right).
\]
where we used the well-known modular identity
\[ \sum_{m \in \BZ} q^{2m^2 + m} = 
\prod_{m \geq 1} \frac{(1-q^{2m})^2}{1-q^m}. \]

We have the isomorphism of moduli spaces
\[
P_{df,0}(K_Y) \cong \Hilb^n([\p^1 \times \BC]/\BZ_2) = \Hilb^n(\p^1 \times \BC)^{\BZ_2},
\]
where $\BZ_2$ acts on $\p^1 \times \BC$ by $(\inv_{\p^1},-1)$ with $\inv_{\p^1} : \p^1 \to \p^1$ an involution with fixed points $0,\infty$.
We can equip $\p^1 \times \BC$ with a $\BZ_2$ equivariant action by the torus $T = \BG_m^2$.
The two fixed points $(0,0)$ and $(\infty, 0)$ can be taken to have torus weights $t,u$ and $t,-u$. By localizing by $T$ and using the above result we obtain:
\begin{align*}
\sum_{d \geq 0} \PT_{df,0} q^d 
& =
\prod_{m \geq 1} \frac{(1-q^{2m})^4}{(1-q^m)^2} \Exp\left( \frac{q^2}{1-q^2} \cdot \left[ \frac{ t u + 2 t + 2 u + 1 }{ (t + 1)(u +1) }+ \frac{ t u^{-1} + 2 t + 2 u^{-1} + 1 }{ (t + 1)(u^{-1} +1) } \right] \right) \\
& = 
\prod_{m \geq 1} \frac{(1-q^{2m})^4}{(1-q^m)^2} \Exp\left( 3 \frac{q^2}{1-q^2} \right) \\
& = \prod_{m \geq 1} \frac{(1-q^{2m})^1}{(1-q^m)^2}
\end{align*}
\end{proof}

Let's relate this computation to Vafa-Witten invariants.
Toda's equation in this case reads
\begin{equation} \label{abc}
\sum_{d} \PT_{df,0} q^d
=
\prod_{r \geq 0, d \geq 0} \exp( (-1)^{r-1} [r]_t \VW(r,df,0) q^d ).
\end{equation}
So taking $\Log^{(2)}$ of \eqref{abc} and using the Proposition we get
\[
\Log^{(2)}( \sum_{d} \PT_{df,0} q^d ) = \sum_{m \geq 1} 2 q^m
= \sum_{r \geq 0} \sum_{d>0} (-1)^{r-1} \vw(r,df,0) [r]_t q^d,
\]
so
\[ 2 = \sum_{r} (-1)^{r-1} \vw(r,df,0) [r]_t. \]
So we find $\vw(r,df,0) = 2$ if $r=1$ and $=0$ if $r \geq 2$,
proving Proposition~\ref{prop:VW computations}(c).

\subsection{Computations II: Points}
We consider the $K$-theoretic DT invariants of the invariants of the Hilbert schemes of points of a quasi-projective Calabi-Yau threefold $X$. By a result of Okounkov \cite{Okounkov} they are given by
\[
\sum_{n} \chi(\Hilb^n(X), \widehat{\CO}) (-p)^n
= \Exp\left( \chi( X, \frac{p \CL_4 ( T_X + K_X - T_X^{\vee} - K_X^{-1}}{(1-p \CL_4)(1-p \CL_5^{-1})}) \right),
\]
where $\CL_4 = \CL_5 = \omega_X^{1/2}$ (we refer to \cite{Okounkov} for the precise notation and requirements on $X$). 

For $X = K_Y$ where $Y$ is an Enrique surface, we have $\omega_X = t^{-1}$, $T_X|_{Y} = T_Y + t$ so
\[
(T_X + K_X - T_X^{\vee} - K_X^{-1})|_{Y}
=
T_Y + \Omega_Y = 0 \quad \in K(Y)\otimes \BQ,
\]
where in the last step we used that $\omega_Y$ is $2$-torsion and $T_Y, \Omega_Y$ are the same class in $K$-theory after tensoring with $\BQ$.
This shows that
\[ \sum_{n} \chi(\Hilb^n(X), \widehat{\CO}) (-p)^n = 1. \]

On the other hand, we have the wall-crossing formula (e.g. \cite{Toda})
\[
\sum_{n} \chi(\Hilb^n(X), \widehat{\CO}) (-p)^n
=
\prod_{n \geq 1} \exp\left( \sum_{n \geq 1} - [n]_t \vw(0,0,n) p^n \right).
\]
So by comparing, we find $\vw(0,0,n)=0$,
proving Proposition~\ref{prop:VW computations}(a).

\subsection{Computations III: Low degree}
\label{subsec:Computations 3 low degree}
The cohomology $H^2(Y,\BZ)$ together with the intersection pairing can be decomposed as 
\[ H^2(Y,\BZ) = U \oplus E_8(-1), \]
where the hyperbolic lattice $U=\binom{0\ 1}{1\ 0}$ has a basis $s,f$ which are smooth elliptic half-fibers of elliptic fibrations on $Y$, so $s^2=f^2=0$ and $s \cdot f=1$.
We let $s_1,s_2 \subset Y$ denote the two curves with $[s_i]=s$, and $f_1, f_2 \subset Y$ the two curves with $[f_i]=f$.

We consider the refined Pandharipande-Thomas invariants in the cases:
\begin{enumerate}
	\item $\beta_1 = s+f$
	\item $\beta_2 = 2s+2f+\alpha$
	\item $\beta_3 = 2s+2f+\gamma$
\end{enumerate}
where $\alpha \in E_8(-1)$ is a class of square $-4$ and $\gamma \in E_8(-1)$ is of square $-2$.
We have $\beta_i^2=2i$.

The K-theoretic Pandharipande-Thomas invariants are invariant under deformations. Thus to compute them it suffices to consider classes in $H^2(Y,\BZ)$ modulo the monodromy group. The monodromy orbits of positive square classes are classified by points in a certain fundamental region of the Weyl group of $U\oplus E_8(-1)$ and are explicitly known, see \cite[Cor.1.5.4]{EnriquesBook}.
For square $2$ classes there is a unique orbit, given by $\beta_1$.
For square $4$ classes there are two orbits, $\beta_2$ and $s+2f$.
For square $6$ classes there are two orbits, $\beta_3$ and $s+3f$.
The linear systems $|s+2f|$ and $|s+3f|$ are more difficult to study since they contain non-reduced curves, such as $s + 2f$. On the other, hand the linear systems $|\beta_i|$ for $i=1,2,3$ contain only reduced curves. The above list hence represents the curve classes of square $2,4,6$ with only reduced members.

\begin{rmk}
A class $\beta \in H^2(Y,\BZ)$ is here taken modulo $2$-torsion. By $|\beta|$ we mean one of the two linear systems of curves in class $\beta$. The linear system we choose depends on the choice of lifting $\beta$ to an integral class $\tilde{\beta}$ (in cohomology with torsion).
However, $\tilde{\beta}$ and $\tilde{\beta}+K_Y$ lie in the same monodromy orbit whenever $2 \nmid \beta$ (see \cite{Knutsen} and \cite[App.C]{Enriques}), so the choice does not matter for our considerations here.
\end{rmk}

\subsubsection{The class $\beta_1=s+f$}
The linear system $|\beta_1|=\p^1$ has precisely 18 singular members, of which 16 are irreducible nodal curve and of which $2$ are reducible consisting of $2$ smooth elliptic curves glued along a point (given by the divisors $s_1+f_1$ and $s_2+f_2$), see \cite[Prop.1.19(ii)]{Sacca}.
Hence we get
\[
\PT_{n,\beta_1} = \Pt_{n,\beta_1} + \sum_{n_1+n_2=n} \Pt_{n_1,s} \Pt_{n_2,f}.
\]
The last term on the right is determined by $\Pt_{n,s} = \Pt_{n,f} = \PT_{n,f} = 2 \delta_{n,0}$.
The first term on the right can be computed through \eqref{toda small}.
By Corollary~\ref{cor:vanishing} there are only two types of possible contributions: $\VW(0,\beta_1,n)$ and $\VW(1,\beta_1,n)$. The contributions $\VW(1,\beta_1,n)$ are determined by the invariants of the Hilbert scheme of points and are known.
On the other hand, 
\[ \VW(0,\beta_1,n) = \chivir_{-t}(M^Y(0,\beta_1,1)). \]
The Hodge numbers of $M^Y(0,\beta_1,1)$ were computed by Sacca \cite{Sacca}. One observes that:
\[
\chivir_{-t}(M^Y(0,\beta_1,1)) = 0.
\]
With $\VW(1,\beta_1,0) = \VW(1,0,-1) = 2 \chivir_{-t}(Y)$
and $\VW(1,\beta_1,1) = \VW(1,0,0)=2$ we find in summary:
\begin{align*}
\sum_{n} \PT_{n,\beta_1} (-p)^n
& = \VW(1,\beta_1,0) + \VW(1,\beta_1,1) [2]_t (p + p^{-1}) + 4 \\
& = (2t^{-1/2} + 2 t^{1/2})p^{-1} + (2 t^{-1} + 24 + 2 t) + (2 t^{-1/2} + 2 t^{1/2})p.
\end{align*}

An alternative way for computing these invariants is as follows.
All curves in class $\beta_1$ are scheme-theoretically supported on the Enriques surface $Y \subset X$. The moduli space of stable pairs is hence isomorphic to two copies of the relative Hilbert scheme $\Hilb^n(\CC/|\beta_1|)$, see \cite{PT_BPS}.
Moreover, since the family of curves in class $\beta_1$ is versal, the relative Hilbert scheme is smooth. It follows, that
\[
\PT_{n,\beta_1} = 2 \chivir_{-t}( \Hilb^{n+1}(\CC/|\beta_1|) ).
\]

The right hand side was computed by G\"ottsche-Shende \cite{GS} for all K-trivial surfaces by using that it universally dependent only on basic intersection numbers of the surfaces and computing the K3 and abelian surface case. Precisely, combining Theorems 2,5,6 of \cite{GS} one obtains an independent computation of $\PT_{n,\beta_1}$ which matches our earlier result.

\subsubsection{The class $\beta_2=2s+2f+\alpha$}
The linear system $|\beta_2| \cong \p^2$ is base-point free and defines a $4:1$ branched cover $Y \to \p^2$. The geometry of this morphism has been intensively studied in the literature, see \cite[Sec.3.4]{EnriquesBook}.

By a straightforward lattice analysis, one checks that all curves in $|\beta_2|$ are reduced. Moreover, they are irreducible, except for 9 curves which are of the form $C_1+C_2$, where $C_i^2=0$ and $C_1 \cdot C_2=2$, i.e. which are given by two smooth elliptic curves glued along $2$ points.

Thus as before
\[
\PT_{n,\beta_2} = \Pt_{n,\beta_2} + 9 \sum_{n_1+n_2=n} \Pt_{n_1,f} \Pt_{n_2,f}
= \Pt_{n,\beta_2} + 36 \delta{n,0}.
\]
The conjectural answer extracted from Theorem~\ref{thm:PT} can then shown to match the G\"ottsche-Shende computation (which is valid if we assume that the family of curves in versal, which is natural to expect). In total
one obtains:
\begin{align*}
\sum_{n}(-p)^n \PT_{n,\beta_2}
& =
(2 s^{-2} + 2 + 2s^{2})(t^{-2}+t^2) + (2 s^{-3} + 22 s^{-1} + 22s + 2s^{3})(t^{-1} + t) \\[-7pt]
& + 2 s^{-4} + 22 s^{-2} + 168 + 22s^{2} + 2s^{4}.
\end{align*}

In particular, since all other Vafa-Witten invariants are known in \eqref{toda small}, the G\"ottsche-Shende computation implies 
$\vw(0,\beta_2,1)=0$, and so
\[ \chi_{-t}(M(0,\beta_2,1)) = 0. \]

\subsubsection{The class $\beta_3=2s+2f+\alpha$}
Here $|\beta_3|\cong \p^3$ is base-point free and defines a double cover of a cubic surface, \cite[Sec.3.3]{EnriquesBook}.
There are the following non-trivial effective splittings of the class $\beta_3 = 2s+2f+\gamma$:
\begin{enumerate}
	\item[(i)] $\beta_3 = s + f + (s+f+\alpha)$, which correspond to curves that are given by three elliptic curves meeting each other in a point,
	\item[(ii)] three splittings $\beta_3 = C_1 + C_2$, where $C_1^2=2$ and $C_2^2=0$ and $C_1 \cdot C_2 = 2$, corresponding to a genus $2$ curve and a genus $1$ curve meeting in $2$ points.
	\item[(iii)] 28 splitting $\beta_3 = C_1 + C_2$, where $C_1^2=0$ and $C_2^2=0$ and $C_1 \cdot C_2 = 3$, corresponding to two elliptic curves meeting in 3 points (this corresponds to $\beta_3 = (s+f+\gamma_1) + (s + f + \gamma_2)$, where $\gamma_i \in E_8(-2)$ have square $-2$ and $\gamma_1 \cdot \gamma_2=1$).
\end{enumerate}
Using that $\pt_{n,f}$ is non-zero only for $n=0$, we obtain that
\[
\PT_{n,\beta_3} =
\Pt_{n,\beta_3}  + \Pt_{0,f}^3 + 3 \Pt_{n,s+f} \Pt_{0,f} + 28 \Pt_{0,f}^2.
\]
We have then checked that the G\"ottsche-Shende computation (which is valid for a versal family of curves, which is natural to expect here) matches the formula in Theorem~\ref{thm:PT}.
From this (modulo this transversality issue) we again obtain
$\chi_{-t}(M(0,\beta_3,1))=0$.

\subsection{Holomorphic anomaly equations} \label{subsec:HAE}
In this last section, we consider the refined holomorphic anomaly equation of our conjectural PT generating series.

Write the PT partition function as
\[ \sum_{\beta \geq 0} \sum_{n \in \BZ} \PT_{n,\beta} (-p)^n Q^{\beta}
= \exp(F) \]
and expand the series
\[ F = \sum_{k \geq 0} Q^{ks} F_{k}, \quad F_k = \sum_{d \geq 0} \sum_{\alpha \in E_8(-1)} N_{ks+df+\alpha} q^{d} \zeta^{\alpha}, \]
where we set $Q^{f}=q$ and $Q^{\alpha}=\zeta^{\alpha}$,
and used the splitting $H^2(Y,\BZ) \cong \BZ s \oplus \BZ f \oplus E_8(-1)$.

We then have that
\[
F_1 = 
2
\frac{\Theta(t, q^2)\Theta(p t^{1/2}, q^2) \Theta(p t^{-1/2}, q^2) \eta(q^2)^8}{\Theta(t,q)\Theta(p t^{1/2},q) \Theta(p t^{-1/2},q) \eta(q)^{16}}
\vartheta_{E_8}(\zeta,q)
\]
where $\vartheta_{E_8}(\zeta,q) = \sum_{\alpha \in E_8} \zeta^{\alpha} q^{\alpha^2/2}$ is the theta-function of the $E_8$-lattice.
The function $F_1$ is a Jacobi form for $\Gamma_0(2)$ of a certain index.
Moreover, $F_k$ for $k>1$ is the $k$-th Hecke transform of $\Gamma_0(2)$ applied to $F_1$, see \cite[Sec.2.5]{Enriques} for definitions and a similar case.

Consider the variable change
\begin{equation} t=e^{\epsilon_1 + \epsilon_2}, \quad p=e^{\frac{\epsilon_1-\epsilon_2}{2}}. \label{var change} \end{equation}
and let us write $\Theta(z)$ instead of $\Theta(e^z)$. Then we get
\[
\left[ F_1 \right]_{\zeta^0} = 
2
\frac{\Theta(\epsilon_1+\epsilon_2, q^2)\Theta(\epsilon_1, q^2) \Theta(\epsilon_2, q^2) \eta(q^2)^8}{\Theta(\epsilon_1+\epsilon_2,q)\Theta(\epsilon_1,q) \Theta(\epsilon_2,q) \eta(q)^{16}}
\]
It is well-known that 
 $\Theta(z) = z \exp(-2 \sum_{k \geq 2} G_k/k! z^k)$, where $G_k$ are the Eisenstein series, see e.g. \cite{Enriques}.
Hence $\Theta(z)$ can be viewed as a power series in $z$ with coefficient of $z^{\ell}$ a weight $\ell-1$ quasi-modular form for $\SL_2(\BZ)$.
The same holds for $\Theta(z,q^2)$, but with the full modular group $\SL_2(\BZ)$ changed to $\Gamma_0(2)$. We hence find that:

\begin{itemize}
\item 
$F_1$ is a power series in $\epsilon_1, \epsilon_2$ with 
coefficients
\[ [ F_1 ]_{\epsilon_1^{a} \epsilon_2^b} = 
f_{a,b}(q) \vartheta_{E_8}(\zeta,q) \frac{\eta(q^2)^8}{\eta(q)^{16}} \]
where $f_{a,b}$ is a quasi-modular form for $\Gamma_0(2)$ of weight $a+b$.
\end{itemize}

The ring of quasi-modular forms for $\Gamma_0(2)$ is the free polynomial ring in $G_2$ over the ring of modular forms,
$\QMod(\Gamma_0(2)) =  \Mod(\Gamma_0(2))[G_2]$.
We hence can speak of the $G_2$-derivative of quasi-modular forms \cite{123}. 
One gets
$\frac{d}{dG_2} \Theta(z,q) = -z^2 \Theta(z,q)$ and $\frac{d}{dG_2} \Theta(z,q^2) = -\frac{1}{2} z^2 \Theta(z,q^2)$.
Inserting this we find that $F_1$ satisfies
the "holomorphic anomaly equation"
\[
\frac{d}{dG_2} F_1 = (\epsilon_1^2 + \epsilon_2^2 + \epsilon_1 \epsilon_2) F_1.
\]
Considering how the Hecke operators interact with $\frac{d}{dG_2}$, see \cite[Prop.2.7]{Enriques}, we get for $k \geq 1$,
\begin{equation} \label{refinedHAE}
\frac{d}{dG_2} F_k = k  (\epsilon_1^2 + \epsilon_2^2 + \epsilon_1 \epsilon_2) F_k.
\end{equation}

Equation \eqref{refinedHAE} can be viewed as a refined holomorphic anomaly equation. It matches the holomorphic anomaly equations for the refined Gromov-Witten invariants as introduced by \cite{BS,Schuler},
and so gives evidence for the refined GW/PT correspondence of \cite{BS}.

\begin{rmk}
We sketch how the holomorphic anomaly equation for the refined Gromov-Witten invariants can be computed. One considers
the Calabi-Yau 5-fold
\[ M = K_Y \times \BC^2 \]
on which a $3$-dimensional torus acts by scaling the fibers on $K_Y$ and by the standard action on $\BC^2$ with weights $t_1, t_2$. The $2$-dimensional subtorus $(\BC^{\ast})^2$ given by the elements $((\lambda_1 \lambda_2)^{-1}, \lambda_1, \lambda_2)$ for $\lambda_1, \lambda_2 \in \BC^{\ast}$ acts Calabi-Yau. We let $\epsilon_i = c_1(t_i)$.
The refined Gromov-Witten potential proposed by Brini-Sch\"uler \cite{BS, Schuler} is
\[
F^{\GW} = \sum_{g} \sum_{\beta} Q^{\beta} \int_{[ \Mbar_{g}(M,\beta) ]^{\vir}} 1.
\]
By definititon it is computed by the virtual localization formula
\[
F^{\GW} = \sum_{g} \sum_{\beta} Q^{\beta} \int_{[ \Mbar_{g}(Y,\beta) ]^{\vir}}
\frac{\BE^{\vee}(-(\epsilon_1+\epsilon_2)) \BE^{\vee}(\epsilon_1) \BE^{\vee}(\epsilon_2)}{-(\epsilon_1 + \epsilon_2) \epsilon_1 \epsilon_2}.
\]
where $\BE^{\vee}(x) = x^g - \lambda_1 x + \ldots + (-1)^g \lambda_g$.

We can consider the partial sums:
\[
F^{\GW}_{g,k} = \sum_{\beta=ks+df+\alpha}
\int_{[ \Mbar_{g}(Y,\beta) ]^{\vir}} \frac{\BE^{\vee}(-(\epsilon_1+\epsilon_2)) \BE^{\vee}(\epsilon_1) \BE^{\vee}(\epsilon_2)}{-(\epsilon_1 + \epsilon_2) \epsilon_1 \epsilon_2}
\]
We apply the holomorphic anomaly equation (HAE) for Enriques surfaces \cite[Thm 4.3]{Enriques}. Similar to \cite[4.4.4]{Enriques} only the second term contributes in the HAE, and gives (in the notation of \cite{Enriques})
\begin{align*}
\frac{d}{dG_2} F^{\GW}_{g,k} & = 2 F_{g_1=1,k_1=0}(\BE^{\vee}(-(\epsilon_1+\epsilon_2)) \BE^{\vee}(\epsilon_1) \BE^{\vee}(\epsilon_2) ; \tau_0(1)) \\
& \quad \quad \times 
F_{g_2=g-1,k_2=k}\left( \frac{\BE^{\vee}(-(\epsilon_1+\epsilon_2)) \BE^{\vee}(\epsilon_1) \BE^{\vee}(\epsilon_2)}{-(\epsilon_1 + \epsilon_2) \epsilon_1 \epsilon_2 }; \tau_0(f) \right) \\
& =
2 \int_Y c_2(Y) \int_{\Mbar_{1,1}} \BE^{\vee}(-(\epsilon_1+\epsilon_2)) \BE^{\vee}(\epsilon_1) \BE^{\vee}(\epsilon_2)
\cdot k F^{\GW}_{g-1,k} \\
& =k (\epsilon_1^2 + \epsilon_2^2 + \epsilon_1 \epsilon_2) F^{\GW}_{g-1,k}.
\end{align*}
Moreover, as conjectured by \cite{BS}, $\exp(F^{\GW}) = \sum_{\beta \geq 0} \sum_{n \in \BZ} \PT_{n,\beta} (-p)^n Q^{\beta}$ under the variable change
	 $t=e^{\epsilon_1+\epsilon_2}$ and $p=e^{\frac{\epsilon_1-\epsilon_2}{2}}$,
so we get a match with \eqref{refinedHAE}.\footnote{Indeed, \cite{BS} define
	$\epsilon_{+} = \frac{\epsilon_1 + \epsilon_2}{2}$, $\epsilon_{-} = \frac{\epsilon_1 - \epsilon_2}{2}$
	$q_{+} = e^{\epsilon_{+}}$, $q_{-} = e^{\epsilon_{-}}$,
	and conjecture that
	\[ \sum_{n} \chi_{\BC^{\ast}_{q_{+}}}( P_{n,\beta}(K_Y), \widehat{\CO}^{\vir} ) (-q_{-})^{n} = \exp(F_{\GW}). \]
	The varible $q_{+}$ is a square root of the equivariant weight of $\omega_{K_Y}$, so $t=q_{+}^2$.}	
\end{rmk}

\section{The case of K3 surfaces} \label{section:K3}
We sketch how our methods for the Enriques surface
can also be applied to K3 surfaces (after accounting for modifications coming from $H^{2,0}(S) \neq 0$).

Let $S$ be a K3 surface.
Let $\VW(r,\beta,n)$ be the refined Vafa-Witten invariant (in the sense of \cite{Thomas}) counting semistable sheaves $F$ on $S \times \BC$ with Chern character $\ch(F) = (r,\beta,n) \in H^{\ast}(S,\BZ)$.
We also let $\PT_{n,\beta}$ be the NO-twisted $K$-theoretic reduced PT invariants of $S \times \BC$.

The results of Section~\ref{subsec:basic results}
have the following analogue for K3 surfaces:
\begin{thm}(Toda's equation) \label{thm:Todas formulaK3} We have:
\[
(-1)^n \PT_{n,\beta}
=
\begin{cases}
-\sum_{r \geq 0} [n+2r]_t \VW(r,\beta,n) & \text{ if } n \geq 0 \\
-\sum_{r > 0} [ |n| + 2r ]_t \VW(r,\beta,|n|) & \text{ if } n<0.
\end{cases}
\]
\end{thm}
\begin{proof}
This again follows immediately from Toda's wall-crossing setup of \cite{Toda}
as soon as we understand how the wallcrossing formula applied to refined reduced invariants on the local K3 surface.
This has been obtained by Thomas \cite{ThomasK3} and so the above result is a direct consequence of his work.

To give some motivation for the final formula here,
let us assume that the wall-crossing formula should hold also "{\em motivically}", i.e. on the level of virtual motives, which would read:
\begin{align*}
\sum_{\beta} \sum_{n \in \BZ} [\PT_{n,\beta}]^{\mathrm{mot}} (-p)^n Q^{\beta}
= & \prod_{\substack{r \geq 0 \\ \beta > 0 \\ n \geq 0 }} \exp\left( -[n+r]_t [\VW(r,\beta,n)]^{\mathrm{mot}} Q^{\beta} p^n \right)  \\
\times & \prod_{\substack{r>0 \\ \beta>0 \\ n>0}} \exp\left( [n+r]_t [\VW(r,\beta,n)]^{\mathrm{mot}} Q^{\beta} p^{-n} \right)
\end{align*}
We want to specialize to the reduced refined $K$-theoretic invariants.
To capture the wall-crossing behaviour of reduced invariants, this specialization should take the form
\[ [\PT_{n,\beta}]^{\mathrm{mot}} \mapsto \epsilon \PT_{n,\beta}^{\mathrm{red}} \]
\[ [\VW_{n,\beta}]^{\mathrm{mot}} \mapsto \epsilon \VW(r,\beta,n). \]
where $\epsilon^2 = 0$. Expending and taking the coefficient of $\epsilon^1$ gives the results.
In other words, the reduced wall-crossing formula boils down to taking the linear term in the wall-crossing formula, ignoring higher quadratic contributions, see \cite{ThomasK3}.
\end{proof}

Similarly, by the arguments of \cite{Toda} we have:

\begin{thm} \label{prop:DT dependence K3}
Let $v=(r,\beta,n) \in H^{\ast}(X,\BZ)$.
\begin{enumerate}
\item[(i)] The invariant $\VW(v)$ does not depend on the choice of polarization used to define it.
\item[(ii)] The invariant $\VW(v)$ depends upon $v$ only through
\begin{itemize}
\item the Mukai square $v^2 = \beta^2 - 2rn - 2r^2$
\item the divisibility $\gcd(r,\beta,n)$
\end{itemize}
\end{enumerate}
\end{thm}

We define the invariants $\vw$ and $\Pt$ as follows:
\begin{gather*}
\PT_{n,\beta} = \sum_{k | (n,\beta)} \frac{(-1)^{n-n/k}}{[k]_t} \Pt_{\beta/k,n/k}(t^k) \\
\VW(r,\beta,n) = \sum_{k|(r,\beta,n)} \frac{1}{[k]_t^2} \vw(r/k, \beta/k, n/k).
\end{gather*}
Toda's equation then can be rewritten:
\begin{equation} \label{bla2}
\Pt_{n,\beta} (-1)^n =
\begin{cases}
-\sum_{r \geq 0} [n+2r]_t \vw(r,\beta,n) & \text{ if } n \geq 0 \\
-\sum_{r > 0} [ |n| + 2r ]_t \vw(r,\beta,|n|) & \text{ if } n<0.
\end{cases}
\end{equation}

We get the following equivalence:
\begin{prop} \label{prop:PT VW for K3}
The following are equivalent:
\begin{enumerate}
	\item[(a)] $\sum_{n} \Pt_{n,\beta} (-p)^n = 
	\left[ 
	\frac{-(t^{1/2} - t^{-1/2})}{\Theta( t, q) \Theta( t^{-1/2} p,q) \Theta( t^{1/2} p,q) \Delta(q)} \right]_{q^{\beta^2/2}}$ \\[2pt]
	\item[(b)] $\vw(r,\beta,n) =
	\left[
	\frac{ (t^{1/2}-t^{-1/2})^2 }
	{ \Theta(t, q)^2 \Delta(q) } \right]_{q^{\beta^2/2-rn-r^2}}$
\end{enumerate}
\end{prop}
\begin{proof}
Assume (b) first. By \eqref{bla2} if (b) holds, then $\vw(r,\beta,n)$ and hence $\Pt_{n,\beta}$ only depends on $\beta$ through its square $\beta^2$.
Hence its enough to prove (a) for $\beta$ running over a choice of curve class $\beta_d$ with $\beta_d^2=2d$ for $d \geq -1$. Then we get
\begin{align*}
\sum_{n,d} \Pt_{n,\beta_d} (-p)^n q^d
& = 
- \sum_{\substack{r,n \geq 0 \\ d \geq -1}} 
[n+2r]_t \vw(r,\beta_d,n) p^n q^d 
- \sum_{r,n>0, d \geq -1} [n+2r]_t \vw(r,\beta_d,n) p^{-n} q^d \\
& = 
\frac{ (t^{1/2}-t^{-1/2})^2 }
{ \Theta(t, q)^2 \Delta(q) } 
\cdot \left( - \sum_{r,n \geq 0} [n+2r]_t p^n q^{rn+r^2} - \sum_{r,n>0} [n+2r]_t p^{-n} q^{rn+r^2} \right) \\
& = 
\frac{ (t^{1/2}-t^{-1/2})^2 }
{ \Theta(t, q)^2 \Delta(q) } 
\cdot 
\frac{-1}{t^{1/2}-t^{-1/2}}
\frac{\Theta(t,q)}{\Theta(t^{1/2}p,q) \Theta(t^{-1/2} p, q)}
\end{align*}
which proves (a). Conversely, similar arguments as in 
part (ii) of the proof of \cite[Proposition 5.16]{Enriques} 
shows that $\vw(r,\beta,n)$ is uniquely determined from $\Pt_{n,\beta}$ through \eqref{bla2},
so if (a) holds, then there is at most one solution for $\vw(r,\beta,n)$ which then must be given by (b).
\end{proof}

\begin{rmk}
Thomas on \cite{ThomasK3} also proves that (a) above implies (b) using a different wall-crossing. He moreover then shows (a) using a double-cosection argument.
\end{rmk}

\end{document}